\documentclass[11pt]{article}
\usepackage[utf8]{inputenc}
\usepackage{lmodern}
\usepackage{subfiles}
\usepackage{enumitem}
	\setenumerate{label={\normalfont(\roman*)}, itemsep=0em} 
        
\usepackage{amsfonts}
\usepackage{amsthm}
\usepackage{amsmath}
\usepackage{amssymb}
\usepackage{amscd}
\usepackage{mathrsfs}
\usepackage{mathtools}
\usepackage{bm}
\usepackage{esint}

\usepackage[margin=3cm]{geometry}
\usepackage{indentfirst}
\usepackage{graphicx}
\usepackage{graphics}
\usepackage{lscape}
\usepackage{tikz-cd}
\usepackage{color}
\usepackage{pict2e}
\usepackage{epic}
\usepackage{epstopdf}
\usepackage{titlesec, titlefoot}
	\titleformat{\section}[block]{\Large\bfseries\filcenter}{\thesection}{1em}{}
\usepackage{commath}
\usepackage{float}
\usepackage{caption}
\usepackage{etoolbox}
\usepackage[affil-it]{authblk}
\usepackage{combelow}

\usepackage[hidelinks]{hyperref}
\hypersetup{bookmarksopen=true} 
\usepackage{hypcap}

\graphicspath{{./Pictures/}}
\allowdisplaybreaks

\usepackage{listings}
\usepackage{fancybox}

\theoremstyle{plain}

\renewcommand*\thesection{\arabic{section}}
\numberwithin{equation}{section} 

\theoremstyle{plain}
\newtheorem{thm}{Theorem}

\newtheorem{lemma}[thm]{Lemma}

\numberwithin{thm}{section} 
\newtheorem{corollary}{Corollary}
\newtheorem{theorem}{Theorem}

\theoremstyle{definition}

\newtheorem{definition}[thm]{Definition}

\newcommand{\thistheoremname}{}
\newtheorem{genericthm}[equation]{\thistheoremname}

\newcommand{\thistheoremnames}{}
\newtheorem*{genericthms}{\thistheoremnames}
\newenvironment{para*}[1]
  {\renewcommand{\thistheoremnames}{#1}%
   \begin{genericthms}}
  {\end{genericthms}}

\expandafter\let\expandafter\oldproof\csname\string\proof\endcsname
\let\oldendproof\endproof
\renewenvironment{proof}[1][\proofname]{%
  \oldproof[\upshape \bfseries #1]%
}{\oldendproof}

\makeatletter
\def\@makechapterhead#1{%
  \vspace*{50\p@}%
  {\parindent \z@ \raggedright \normalfont
    \interlinepenalty\@M
    \Huge\bfseries  \thechapter.\quad #1\par\nobreak
    \vskip 40\p@
  }}
\makeatother

\newcommand{\reqnomode}{\tagsleft@false}
\def \d{\,{\rm d}}

\allowdisplaybreaks
\makeatletter
\DeclareRobustCommand*{\bfseries}{%
  \not@math@alphabet\bfseries\mathbf
  \fontseries\bfdefault\selectfont
  \boldmath
}

\makeatother

\newlength{\defbaselineskip}
\newcommand\eps\varepsilon
\def\mean#1{\mathchoice%
          {\mathop{\kern 0.2em\vrule width 0.6em height 0.69678ex depth -0.58065ex
                  \kern -0.8em \intop}\nolimits_{\kern -0.4em#1}}%
          {\mathop{\kern 0.1em\vrule width 0.5em height 0.69678ex depth -0.60387ex
                  \kern -0.6em \intop}\nolimits_{#1}}%
          {\mathop{\kern 0.1em\vrule width 0.5em height 0.69678ex
              depth -0.60387ex
                  \kern -0.6em \intop}\nolimits_{#1}}%
          {\mathop{\kern 0.1em\vrule width 0.5em height 0.69678ex depth -0.60387ex
                  \kern -0.6em \intop}\nolimits_{#1}}}

\numberwithin{equation}{section}
\allowdisplaybreaks[4]

\def\eqn#1$$#2$${\begin{equation}\label#1#2\end{equation}}
\newcommand\R{\mathbb{R}}

\newcommand{\Div}{\mathrm{Div}}
\newcommand{\F}{\mathscr F}

\def \tp{\textup}
\def \p{\partial}
\def \e{\varepsilon}
\def \D{\mathrm{D}}

\def \M{\mathbb M}
\def \LL{\mathrm L}

\def \WW{\mathrm{W}}
\def \BV{\mathrm{BV}}
\def \wstar {\overset{\ast}{\rightharpoonup}}

\newcommand\restr[2]{{
  \left.\kern-\nulldelimiterspace 
  #1 
  \vphantom{|} 
  \right|_{#2} 
  }}

\title{On the validity of the Euler-Lagrange system without growth assumptions}
\author[1]{Lukas Koch}
\author[2]{Jan Kristensen}

\affil[1]{\small Max Planck Institute for Mathematics in the Sciences, 04103 Leipzig,Germany
\protect \\
  {\tt{lkoch@mis.mpg.de}}
  \vspace{1em} \ }

\affil[2]{\small University of Oxford, Andrew Wiles Building Woodstock Rd, Oxford OX2 6GG, United Kingdom 
\protect \\
  {\tt{kristens@maths.ox.ac.uk}}
  \vspace{1em} \ }

\usepackage{etoolbox}
\makeatletter
\patchcmd{\@adjustvertspacing}
  {\jot\baselineskip \divide\jot 4}
  {\jot=.5\baselineskip}
  {}{}
\makeatother

\makeatother
\begin{document}
\maketitle

\begin{abstract}
The constrained minimisers of convex integral functionals of the form
\begin{align*}
\F(v)=\int_\Omega F(\nabla^k v(x))\d x
\end{align*}
defined on Sobolev mappings $v\in \WW^{k,1}_g(\Omega , \R^N )\cap K$, where $K$ is a closed convex subset of the Dirichlet class  $\WW^{k,1}_{g}(\Omega , \R^N )$ are
characterised as the energy solutions to the Euler-Lagrange inequality for $\F$. We assume that the essentially smooth integrand
$F\colon \R^{N} \otimes \odot^{k}\R^{n} \to \R\cup\{+\infty\}$ is convex, lower semi-continuous, proper and at least super-linear at infinity.
In the unconstrained case $K=\WW^{k,1}_{g}(\Omega , \R^N )$, if the integrand $F$ is convex, real-valued, and satisfies a demi-coercivity condition, then
$$
\int_{\Omega} \! F^{\prime}(\nabla^{k} u) \cdot \nabla^{k}\phi \, \d x =0
$$
holds for all $\phi \in \WW_{0}^{k}( \Omega , \R^{N})$, where $\nabla^{k} u$ is the absolutely continuous part of the vector measure $D^{k}u$.
\end{abstract}
\section{Introduction and results}
Let $\Omega$ be an open and bounded subset of $\R^n$, $g\in \WW^{k,1}(\Omega) \equiv \WW^{k,1}(\Omega,\R^N )$, and $K$ a closed convex subset on $\WW^{k,1}( \Omega )$. 
We consider the functional
\begin{align}\label{eq:def}
\F(v)=\int_\Omega F(\nabla^k v)\d x
\end{align}
defined on $K\cap \WW^{k,1}_g(\Omega)$, where $\WW^{k,1}_{g}( \Omega ) \equiv g+\WW^{k,1}_{0}( \Omega )$ is the Dirichlet class determined by $g$.
Denoting $\mathbb{M} \equiv \R^{N} \otimes \odot^{k}\R^{n}$, we assume that $F \colon \mathbb{M} \to \R \cup \{ +\infty \}$ is a convex, lower semicontinuous
extended real-valued integrand satisfying moreover the coercivity condition
\begin{align}\label{eq:superlinearity}
F(\xi)\geq \theta(|\xi|)\tag{H1}
\end{align}
for all $\xi\in \M$, where $\theta\colon[0,\infty)\to [0,\infty)$ is an increasing convex function satisfying
\begin{align*}
\frac{\theta(t)}{t}\to \infty \text{ as } t\to\infty.
\end{align*}
We will also be interested in relaxing the coercivity condition \eqref{eq:superlinearity} to demi-coercivity, namely
\begin{align}\label{eq:demiCoerc}
F(\xi)\geq a(\xi)+c_1|\xi|-c_2\tag{H2}
\end{align}
for some linear $a$ and some constants $c_1>0$, $c_2\in \R$.

We remark that under \eqref{eq:superlinearity} the pointwise definition \eqref{eq:def} agrees with relaxed definitions of $\F(v)$ in the
sense of Lebesgue-Serrin-Marcellini type definitions, see e.g. \cite{Fonseca1997}. Further, the existence of minimisers in this set-up follows
from the direct method. Under \eqref{eq:demiCoerc}, we need to work with a Lebesgue-Serrin-Marcellini type relaxed version of the functional and consider
minimisers in $\BV^k(\Omega) \equiv \BV^{k}( \Omega , \R^{N})$, consisting of integrable maps $u \colon \Omega \to \R^{N}$ whose distributional derivatives
up to and including $k$-th order are bounded Radon measures on $\Omega$. To be precise, we will define:
Let $\Omega\Subset\Omega'$ and $g\in W^{k,1}(\Omega')$ such that $\int_{\Omega'\setminus\Omega} F(\nabla^k g)\d x<\infty$. We then define for $u\in \BV^k(\Omega)$,
\begin{align*}
\F(u)&=\inf\big\{\liminf \int_{\Omega'}F(\nabla^k u_j)\d x\colon (u_j)\in X\big\}\\
X&=\{(u_j)\colon u_j\in K\cap\WW^{k,1}(\Omega'),\nabla^i u_j\wstar \nabla^i u \text{ in } \BV(\Omega) \text{ for } i=0,\ldots,k-1,\\
&\qquad u_j\to g \text{ in } \WW^{k,1}(\Omega'\setminus\Omega)\}.
\end{align*}
Note that if $u\in K\cap\WW^{k,1}_g(\Omega)$, then $\F(u)=\int_\Omega F(\nabla^k u)\d x+\int_{\Omega'\setminus\Omega} F(\nabla^k g)\d x$, so that $\F$ provides
a suitable extension to $\BV^k(\Omega)$.

As the precise definitions of extremal and minimiser are important in this paper, we recall the relevant them here.
\begin{definition}
A mapping $u\in \BV^k(\Omega)$ is a minimiser of $\F$ if $\F(u)\leq \F(v)$ for all ${v\in \BV^k(\Omega)}$. If $F$ satisfies \eqref{eq:superlinearity}
(and we set $\F(v)=\infty$ if $v\in \BV^k(\Omega)\setminus \WW^{k,1}(\Omega)$), this is equivalent to saying that a mapping ${u\in \WW^{k,1}_g(\Omega)\cap K}$
is a minimiser if $F(\nabla^k u)\in \LL^1(\Omega)$ and
$$
\int_\Omega F(\nabla^k u)\d x\leq \int_\Omega F(\nabla^k v)\d x
$$
for any $v\in \WW^{k,1}_g(\Omega)\cap K$.
\end{definition}
\begin{definition}
A mapping $u\in \WW^{k,1}_g(\Omega)\cap K$ is an energy-extremal if $F'(\nabla^k u)\in \LL^1(\Omega)$ satisfies
$$
\int_\Omega \big\langle F'(\nabla^k u),\nabla^k(v-u)\big\rangle \d x\geq 0
$$
for any $v\in K\cap \WW^{k,\infty}_g(\Omega)$. Note that this entails that $\langle F'(\nabla^k u),\nabla^k u\rangle \in \LL^1(\Omega)$. Here $\nabla^k$ denotes the absolutely continuous part with respect to Lebesgue measure of the gradient. If $u\in W^{k,1}(\Omega)$, this agrees with the usual definition of the gradient.
\end{definition}

In our set-up, where we only consider convex autonomous integrands, it is easy to see that energy-extremals must be minimisers. The reverse question is considerably more difficult and the answer is delicate already in the case of convex, real-valued integrands satisfying $p$-growth. We do not discuss the issue of integrands with $(p,q)$-growth further here but refer to \cite{Serrin1964,Iwaniec1994,Lewis1993,Mingione2006,Mingione2021,Koch2020,
Koch2021a} for further discussion and references. The question for convex integrands without growth conditions has been studied in \cite{Carozza2015}. Our results here extend and strengthen the results of \cite{Carozza2015} in several directions. Whereas \cite{Carozza2015} considered convex real-valued integrands of first order with superlinear growth at infinity, we consider convex extended real-valued integrands of $k$-th order with superlinear growth at infinity and moreover allow to restrain the values of the integrand within a weakly  closed subset of $\WW^{1,1}(\Omega)$. Further, we are able, in the unconstrained set-up, to relax the super-linear growth to the linear bound \eqref{eq:demiCoerc}.

Considering constrained integrands with super-linear growth, our main result is a direct translation of the main theorem in \cite{Carozza2015} generalizing it simultaneously to
extended-realvalued integrands.
\begin{theorem}\label{thm:mainConstrIntro}
Suppose $F\colon \M\to \R\cup\{ +\infty\}$ is convex, lower semi-continuous, proper, essentially smooth and satisfies \eqref{H1Constr}.
Let $g\in \WW^{k,1}(\Omega)\cap K$ with $F(s\nabla^k g)\in \LL^1(\Omega)$ for some $s>1$. Then minimisers $u\in \WW^{k,1}_g(\Omega)\cap K$ are characterised by the conditions
\begin{align*}
F^\ast(F'(\nabla^k u)) \in \LL^1(\Omega),\qquad \langle F'(\nabla^k u),\nabla u)\in \LL^1(\Omega)
\end{align*}
and
\begin{align*}
\int_\Omega F'(\nabla^k u)\nabla^k(u-v)\geq 0 \text{ for all } v\in K \text{ such that } u-v\in\WW^{k,\infty}_0(\Omega).
\end{align*}
\end{theorem}

Our proof follows essentially along the same lines as \cite{Carozza2015}. We approximate $F(\cdot)$ from below by a sequence of integrands $F_k(\cdot)$.
The idea is that minimisers $u_k$ of the regularised functionals should have the desired properties, converge to the minimiser $u$ of $\F(\cdot,\Omega)$
and that these properties are retained in the limit. Our approximations are smooth, convex and globally Lipschitz-continuous.
The corresponding variational problems are degenerately convex problems of linear-growth. It is known that in general such problems need to be solved in $BV(\Omega)$.
For integrands satisfying superlinear growth, we avoid the use of $BV(\Omega)$ by utilising Ekeland's variational principle.
Convergence of $u_k$ to $u$ in the $\LL^1$-sense is established using the theory of Young measures.

The key tool driving our arguments is convex duality theory and it is our more careful use of this theory that allows us to extend to essentially smooth integrands.
Duality ideas in the context of the calculus of variations already appear in \cite{Zhikov1994} and were used in the context of integrands with linear growth in \cite{Seregin1993}.
In recent developments, the idea has been underused in the opinion of the authors, but has nevertheless been applied in the context of standard growth \cite{Carozza2005,Carozza2006,Carozza2013}
and faster than exponential growth \cite{Bonfanti2012,Bonfanti2013}.

We remark that there is an extensive literature on functionals with nonstandard growth and refer to the surveys \cite{Mingione2006,Mingione2021} for a general exposition
and more references. Further information, regarding in particular functionals with linear and nearly-linear growth can be found in \cite{Fuchs2000,Bildhauer2003}. We would like to point out that the energy-extremality of minimisers plays a key role in the approach to these problems. Finally, we remark that in
the case of non-convexity of the integrand the situation is considerably more complicated and the results much weaker, see e.g. \cite{Kristensen2006,Kristensen2010,Gmeineder2019}.

In the case of linear growth \eqref{eq:demiCoerc}, while still proceeding along the same lines of argument, the proof requires more care than in the super-linear case.
The key observation is a representation formula that seems to have gone unnoticed in the literature so far.
\begin{theorem}\label{thm:repr}
Suppose $\Omega$ is a Lipschitz domain. Assume that $F\colon\M\to \R$ is $C^1$, convex and satisfies \eqref{eq:demiCoerc}. Then for $u\in\BV^k(\Omega)$ it holds that
\begin{align*}
\F(u)=\int_{\Omega'}F(\nabla^k u)\d x+\int_{\Omega'} F(\D_s^k u)
\end{align*}
where we decompose the gradient of $u$ into its absolutely continuous and singular parts $\D u=\nabla u+ \D_s u$.
\end{theorem}

Using this representation we establish the following result:
\begin{theorem}
Let $\Omega$ be a Lipschitz domain. Suppose that $F$ is $C^1$, strictly convex and satisfies \eqref{eq:demiCoerc}. Let $u\in \BV^k(\Omega)$ be a minimiser of $\F(\cdot)$ in the unconstrained setting where $K=W^{k,1}(\Omega)$. Then $F'(\nabla^k u)$ is divergence-free in the sense of distributions. Moreover $\langle F'(\nabla^k u),\nabla^k u\rangle\in L^1(\Omega)$
\end{theorem}

The paper is structured as follows. In Section \ref{sec:prelim} we define our notation and recall some facts from convex duality theory. We establish the representation formula of Theorem \ref{thm:repr} in Section \ref{sec:approx}. The constrained superlinear case is treated in Section \ref{sec:constr}, while the linear case can be found in Section \ref{sec:BV}.

\section{Preliminaries}\label{sec:prelim}
Throughout this paper we denote by $c$ a general constant that may vary from line to line. We denote the standard norm on $\R^n$ by $|\cdot|$ and we utilise the same norm on $\M\coloneqq\R^{N\times n\times k}$ so that for matrices $\xi,\eta\in \M$ we write $\langle\xi,\eta\rangle = \tp{trace}(\xi^T \eta)$ for the usual inner product and $|\xi|=\langle \xi,\xi\rangle^\frac 1 2$ for the corresponding inner product. We denote for $a\in \R^n$, $b\in \M$ by $a\otimes b\in \M$ the usual tensor product and note that $(a\otimes b)x=(b\cdot x)a$ for $x\in \R^n$ and that $|a\otimes b|=|a||b|$.

If $F\colon \M\to \R$ is continuously differentiable at $\xi\in \M$ we write for $\eta\in \M$
$$
F'(\xi)\eta=\frac{\d}{\d t}\restr{}{t=0}F(\xi+t\eta)=\langle F'(\xi),\eta\rangle.
$$
Note that here we view $F'(\xi)$ both as an $N\times n\times k$ matrix and as the corresponding linear form on $\M$.

For $1\leq p\leq\infty$, we denote by $\WW^{k,p}(\Omega)=\WW^{k,p}(\Omega,\R^N)$ the usual Sobolev space. $\BV^k(\Omega)=\BV^k(\Omega,\R^N)$ denotes the space of functions with derivatives up to order $k$ that are functions of bounded variation. 

Given $g\in \WW^{1,1}(\Omega)$ we write $u\in \WW^{1,1}_g(\Omega)$ if $u-g\in \WW^{1,1}_0(\Omega)$, where the latter is defined as the closure of the space of smooth compactly supported test maps $C^\infty_c(\Omega,\R^N)$ in $\WW^{1,1}(\Omega)$. We remark that due to Mazur's lemma, weakly closed sets in $W^{1,1}(\Omega)$ are strongly closed and hence we refer to them simply as closed.

If $u\in \BV^k(\Omega)$, we write $\D^k u = \nabla^k u + \D_s^k u$ where $\nabla^k u$ is the absolutely continuous part with respect to Lebesgue measure of $\D^k u$ and $\D_s u$ denotes the singular part.

Our reference regarding convex analysis is \cite{Rockafellar1970}, but for the readers convenience we recall some of the key fact we use here. Throughout this discussion $f\colon \M\to \R\cup \{\infty\}$ will be a convex function. We denote the \textit{domain} of $f$ by
\begin{align*}
\tp{dom}(f) = \{z\in\M\colon f(z)<\infty\}.
\end{align*}
Note that $\tp{dom}(f)$ is an open convex set.

$f$ is called \textit{proper} if $\tp{dom}(f)$ is non-empty and $f$ is finite on it.

We say that $f$ is \textit{essentially smooth} if the following assumptions are satisfied
\begin{enumerate}
\item $C=\tp{dom}(f)\neq \emptyset$
\item $f$ is continuously differentiable in $C$
\item $\lim_{i\to\infty} |\nabla f(x_i)|=+\infty$ whenever $(x_i)\subset C$ converges to a boundary point $x$ of $C$.
\end{enumerate}

Note that a convex function $f$, finite on an open convex set $C$, that is differentiable on $C$, is continuously differentiable on $C$. In particular, if $f$ is essentially smooth, then $f$ is continuously differentiable on $\tp{dom}(f)$.

We say a proper convex function $f$ is \textit{essentially strictly convex} if $f$ is strictly convex on every convex subset of $\tp{dom }\p f$. Here $\p f$ is the subdifferential of $f$.

We have that
\begin{theorem}
A lower semi-continuous proper convex function $f$ is essentially strictly convex if and only if its conjugate $f^\ast$ is essentially smooth.
\end{theorem}

Moreover we make the following observation.
\begin{lemma}\label{lem:linear}
Suppose $f\colon \M\to \R\cup\{\infty\}$ is convex. Then $\tp{dom}(f)$ has interior points if and only if $f^\ast$ is demi-coercive, that is there exists a linear function $L$ and constants $c_1>0$ and $c_2\in \R$ such that 
$$
f^\ast(\cdot)+L(\cdot)\geq c_1 |\cdot|+c_2.
$$
In fact, $B_r(x_0)\subset \tp{dom}(f)$ if and only if for some $c\in \R$,
\begin{align}\label{eq:demiCoercivity}
f^\ast(\xi)-\langle x_0, \xi\rangle \geq r |\xi|+c.
\end{align}
for all $\xi\in \M$.
\end{lemma}
\begin{proof}
Note that $f$ is necessarily continuous in $\tp{dom}(f)$, so that the supremum in \eqref{eq:demiCoercivity} is well-defined and finite. Assume first that $B_r(x_0)\subset \tp{dom}(f)$. Then for $\xi\in \R^{N\times n}$,
\begin{align*}
f^\ast(\xi) =& \sup_{z\in \R^{N\times n}} \langle \xi,z\rangle - f(z)\\
\geq& \left\langle \xi,x_0+r \frac \xi {|\xi|}\right\rangle-f\left(x_0+r \frac \xi {|\xi|}\right)\\
\geq& \langle \xi,x_0\rangle +\rho |\xi|+\inf_{z\in B_r(x_0)} f(z).
\end{align*}
Conversely, if \eqref{eq:demiCoercivity} holds, we have for $0<\tau<r$ and $\theta\in R^{N\times n}$ with $|\theta|=1$,
\begin{align*}
f^{\ast\ast}(x_0+\tau\theta) =& \sup_{z\in \R^{N\times n}} \langle x_0+\tau\theta,z\rangle- f^\ast(z)\\
\leq& \sup_{z\in \R^{N\times n}} \langle x_0+\tau\theta,z\rangle - r|z|-\langle x_0,z\rangle -c\\
\leq & \langle |\tau||\theta||z|-r|z|+c\\
\leq& c.
\end{align*}
In particular, we deduce that $B_r(x_0)\subset \tp{dom}(f)$.
\end{proof}

Finally we record a technical lemma that allows us to pull in the boundary of a Lipschitz domain. The idea is to replicate in this setting the behaviour of the map $x\to s x$ on the unit ball.
\begin{lemma}\label{lem:diffeos}
Suppose $\Omega\Subset\Omega'$ is a Lipschitz domain. Then there is a family of Lipschitz-diffeomorphisms $\Psi_s\colon \Omega'\to\Omega'$ such that
\begin{enumerate}
\item $J\Psi_s\to 1$ and $|D\Psi_s-\tp{Id}|\to 0$ uniformly in $\Omega'$ as $s\nearrow 1$.
\item $\Psi_s(\Omega)=\Omega^s\Subset\Omega$, where $|\Omega\setminus\Omega^s|\sim d(\p\Omega,\p\Omega\setminus\Omega^s)$.
\item $\Psi_s=\tp{Id}$ in an open neighbourhood of $\p\Omega'$
\end{enumerate}
\end{lemma}
\begin{proof}
Let $X\in C^\infty(\R^n,\R^n)$ be 
a smooth vector field transversal to $\p\Omega$. Fix $t_0>0$.
Given $z\in\p\Omega$ and for $|t|\leq 2t_0$ consider the flow
\begin{align*}
\frac{d h_z}{d t} =& X(h(t))\\
h_z(s)=& z
\end{align*}
and set $\Psi(z,t)=h_z(t)$. After possibly reducing the value of $t_0$, the maps $\Psi,\Psi^{-1}$ are Lipschitz-regular diffeomorphisms on a neighbourhood of $\p\Omega$ which we denote $V$. Moreover the Jacobians of $\Psi$, $\Psi^{-1}$ are bounded. 

Let $\frac 1 2\leq s<1$. Consider ${\tau_s\colon [-t_0,t_0]\to [-t_0,t_0]}$, a sequence of strictly monotonically increasing smooth maps with
\begin{align*}
\tau_s(-t_0)=-t_0 & & \tau_s(0)=-(1-s)t_0 & &\tau_s(t)=t \text{ for } t\in(t_0/2,t_0)
\end{align*}
and such that $\tau_s\to \tp{Id}$ in $C^2([-t_0,t_0])$ as $s\to 1$. Define
\begin{align*}
\Psi_s(x)=\begin{cases}
			\Psi\left(x_0,\tau_s(t)\right) &\text{ for } x=\Psi(x_0,t)\in V \\
			x &\text{ else }. 
		  \end{cases}
\end{align*}
Using the chain rule we note that $\Psi_s\colon \Omega'\to \Omega'$ is a Lipschitz-regular diffeomorphism. Denote its inverse by $\Psi_s^{-1}\colon \Omega'\to\Omega'$ and note that using the Inverse Function Theorem and the chain rule, $\Psi_s^{-1}$ is also Lipschitz. Further by direct calculation, $\Psi_s\to \tp{Id}$ in $C^{1}(\Omega')$ as $s\nearrow 1$. In particular, also $J\Psi_s\to 1$ uniformly in $\Omega'$ as $s\nearrow 1$. 

Finally, we remark that $|\Omega\setminus\Omega^s|\sim d(\p\Omega,\p\Omega\setminus\Omega^s)$ where the implicit constant depends only on $\p\Omega$ and $n$.
\end{proof}

\section{Approximations and a representation formula}\label{sec:approx}
We assume that for $\xi\in \M$ and some $c>0$,
\begin{align}\label{eq:growth}
F(\xi)\geq c|\xi|.
\end{align}
Note that (after possibly adjusting $F$ by adding an affine function) \eqref{eq:growth} encapsulates both \eqref{eq:superlinearity} and \eqref{eq:demiCoerc}.

The main goal of this section is to prove the representation formula of Theorem \ref{thm:repr}.
In order to prove Theorem \ref{thm:repr} we need to construct appropriate approximations of $\F(\cdot)$.
\begin{lemma}\label{lem:approximations}
Suppose $\Omega$ is a Lipschitz domain. Assume that $F\colon\Omega\to\R\cup\{+\infty\}$ is convex, essentially smooth. Then there exists a sequence of smooth, convex integrands $\{F_j\}$ with linear growth that satisfy \eqref{eq:growth}, $F_j\nearrow F$, $F_j^\infty\nearrow F^\infty$ pointwise as $s\nearrow \infty$. Moreover, $F_j\to F$ locally uniformly on $\tp{dom}(F)$ and $F'_j\to F'$ locally uniformly on $\tp{dom}(F)$ as $s\nearrow \infty$.
\end{lemma}
\begin{proof}
Since $F$ is essentially smooth, $\tp{int}(F)$ is non-empty.
By changing coordinates if necessary, we may assume that $0$ is an interior point of $\tp{dom}(F)$, that is there is $r>0$ such that $B_r(0)\subset \tp{int}(F)$.

Introduce the Fenchel-conjugate of $F$,
\begin{align*}
F^\ast(z)=\sup_{\xi\in \M} \left(\langle \xi,z\rangle-F(\xi)\right).
\end{align*}
Note that, by Lemma \ref{lem:linear}, we have for some $c\in \R$,
\begin{align}\label{eq:lower}
F^\ast(\xi)\geq r|\xi|+c.
\end{align}
Finally, since $F(\cdot)$ is essentially smooth, $F^\ast(\cdot)$ is essentially strictly convex.

For each $j>0$ and $\xi\in \M$ define
\begin{align*}
\overline F_j^{\ast\ast}(\xi):=\max_{|z|\leq j} \left(\langle \xi,z\rangle-F^\ast(z)\right).
\end{align*}
Note that this is a real-valued, convex, and globally $j$-Lipschitz function. Since $F$ is lower semi-continuous and convex, we have that
\begin{align}\label{eq:1}
\overline F_j^{\ast\ast}(\xi)\nearrow F^{\ast\ast}(\xi)=F(\xi) \text{ as } j\nearrow \infty.
\end{align}
Moreover, it is straightforward to check that $\overline F_j^{\ast\ast}(\xi)\geq \theta(\xi)$.

We denote by $\Phi$ the standard, radially symmetric, and smooth convolution kernel and set, for $\e>0$, $\Phi_\e(\xi)=\e^{-\tp{dim}\M}\Phi(\e^{-1}\xi)$ for $\xi\in \M$. Note that since $F_j$ is convex and $j$-Lipschitz
\begin{align*}
\overline F_j^{\ast\ast}(\xi)\leq \Phi_\e\star \overline F_j^{\ast\ast}(\xi)\leq \overline F_j^{\ast\ast}(\xi)+j\e.
\end{align*}
For integers $j>0$ and sequences $(\delta_j),(\mu_j)\subset(0,\infty)$, which we specify at a later point, define
\begin{align*}
F_j(\xi)= \Phi_{\delta_j}\star \overline F^{\ast\ast}(\xi)-\mu_j.
\end{align*}
Clearly $F_j$ is convex and $j$-Lipschitz. For $\xi\in \M$ and $j>1$ we estimate:
\begin{align*}
F_j(\xi)\leq& \overline F_j^{\ast\ast}(\xi)+j \delta_j-\mu_j\\
\leq& \overline F^{\ast\ast}_{j+1}(\xi)+j\delta_j-\mu_j\\
\leq& (\Phi_{\delta_{j+1}}\star \overline F^{\ast\ast}_{j+1})(\xi)+j\delta_j-\mu_j\\
\leq& F_{j+1}(\xi)+\mu_{j+1}+j\delta_j-\mu_j\\
\leq& F_{j+1}(\xi)
\end{align*}
with the choice
\begin{align*}
\delta_j = \frac 1 {j^3} \quad \mu_j = \frac 1 {j-1}.
\end{align*}
In particular, $F_j(\xi)\nearrow F(\xi)$ as $j\nearrow \infty$ pointwise in $\xi$. By Dini's lemma, the convergence is locally uniform on $\tp{dom}(F)$.

We note that for $\xi\in \M$,
\begin{align*}
\sup_j F_j^\infty(\xi)=\sup_{s,t} \frac{F_j(t\xi)}{t} = \sup_t \frac{F(t\xi)}{t} = F^\infty(\xi),
\end{align*}
so that, since clearly $F_j^\infty(\xi)\leq F_{j+1}^\infty(\xi)$, $F_j^\infty(\xi)\nearrow F^\infty(\xi)$ pointwise in $\xi$.

We next show that $F'_j(\xi)\to F'(\xi)$ locally uniformly on $\tp{dom}(F)$. In order to see this, assume that $(\xi_j)\subset \tp{dom}(F)$ with $\xi_j\to \xi\in \tp{dom}(F)$. Consider $(F_j'(\xi_j))$. As difference-quotients of convex functions are increasing in the increment, we have for all $\eta\in \M$ and $0<t\leq 1$,
\begin{align*}
|\langle F'_j(\xi_j)-F'(\xi),\eta\rangle|\leq& \frac{F_j(\xi_j+t\eta)-F_j(\xi_j)-\langle F'(\xi),t\eta\rangle}{t}|\\
\leq&  |F_s(\xi_j+\eta)-F_j(\xi_j)-\langle F'(\xi),\eta\rangle|.
\end{align*}
Consequently, we find
\begin{align*}
\limsup_{j\to\infty} \frac{|\langle F'_j(\xi_j)-F'(\xi),t\eta\rangle|}{t}\leq \left|\frac{F(\xi+t\eta)-F(\xi)}{t}-\langle F'(\xi),\eta\rangle\right|.
\end{align*}
Letting $t\to 0$, the right-hand side vanishes, proving the asserted locally uniform convergence.
\end{proof}

We are now ready to prove Theorem \ref{thm:repr}.
\begin{proof}[Proof of Theorem \ref{thm:repr}]
Let $u\in \BV^k(\Omega)$. There exists a sequence $(u_j)\subset \WW^{k,1}(\Omega)$ with $u_j\to g$ in $\WW^{k,1}(\Omega'\setminus\Omega)$ such that
\begin{align*}
\int_{\Omega'} F(\nabla^k u_j)\d x\to \F(u).
\end{align*}
Due to \eqref{eq:demiCoerc}, $(u_j)$ is bounded in $W^{k,1}(\Omega')$ and hence we may extract a subsequence such that $\nabla^k u_j \overset{Y^\ast}{\rightharpoonup} \nu$ in the sense of Young measures. We then have, c.f. Proposition 3.3. in \cite{Kristensen2020},
\begin{align*}
\F(u)\geq \int_{\Omega'}\big\langle \nu_x,F_j\big\rangle \d x+\int_{\Omega'}\big\langle \nu_x^\infty,F_j^\infty\big\rangle\d x
\end{align*}
for all $j$, so that, by monotone convergence,
\begin{align*}
\F(u)\geq \int_{\Omega'}\big\langle \nu_x,F\big\rangle \d x+\int_{\Omega'}\big\langle \nu_x^\infty,F^\infty\big\rangle\d x.
\end{align*}
By Jensen's inequality,
\begin{align*}
\int_{\Omega'}\big\langle \nu_x,F\big\rangle \d x+\int_{\Omega'}\big\langle \nu_x^\infty,F^\infty\big\rangle\d \lambda\geq& \int_{\Omega'}F(\overline\nu_x)\d x+\int_{\Omega'} F^\infty(\overline\nu_x^\infty)\d\lambda\\
=& \int_{\Omega'} F(\overline\nu_x)+F^\infty(\overline\nu_x^\infty)\frac{\d\lambda}{\d x}\d x+\int_{\Omega'} F^\infty(t\overline\nu_x^\infty)\d\lambda^s\\
\geq&\int_{\Omega'} F\left(\overline\nu_x+\overline\nu_x^\infty\frac{\d\lambda}{\d x}\right)\d x+\int_{\Omega'} F^\infty(\overline\nu_x^\infty)\d\lambda^s.
\end{align*}
To obtain the last line we used the convexity of $F$.

Let $\{\Psi_t\}$ denote the family of Lipschitz-diffeomorphism of Lemma \ref{lem:diffeos}.
We denote by $u_{\e,t}=\Phi_\e\star u^t$ mollification of $u^t$ with the standard mollifier $\Phi_\e$ where
$$
u^t(x)=u(\Psi_t(x)).
$$

Note that, if $\e <d(\p\Omega,\Psi_t(\Omega))$, $u_\e\in W^{k,1}(\Omega')$. Moreover $u_{\e,t}\to u_t$ in $W^{k,1}(\Omega'\setminus\Omega)$ as $\e\searrow 0$, $u_t\to u$ in $W^{k,1}(\Omega'\setminus\Omega)$ as $t\nearrow 1$ and $\nabla^k u_{\e,t}\wstar \nabla^k u_t$ in $BV(\Omega)$ as $\e\to 0$, $\nabla^k u_t\wstar \nabla^ k u$ in $BV(\Omega)$ as $t\nearrow 1$.
Using again the convexity of $F$ we find
\begin{align*}
\int_{\Omega'}F(\nabla^k u_{\e,t})\d x\leq& \int_{\Omega'} \Phi_\e \star \left(F(\nabla^k u^t)\d x+F^\infty(\D_s^k u^\delta)\right)\\
\overset{\e\to 0}\rightarrow& \int_{\Omega'} F(\nabla^k u^t)\d x+F^\infty(\D_s^k u^t)\\
=& \int_{\Omega'\setminus\Omega} F(\nabla^k g)|\tp{det}\nabla \Psi_t|+\int_{\Omega} F(\nabla^k u)|\tp{det}\nabla \Psi_t|\d x\\
&\qquad+\int_{\Omega'} F^\infty(\D_s^k u)|\tp{det}\nabla \Psi_t|\\
\overset{\delta\to 0}\rightarrow& \int_{\Omega'}F(\nabla^k u)\d x+F^\infty(\D_s^k u)
\end{align*}

Consequently,
\begin{align}\label{eq:repr}
\F(u)=\int_{\Omega'} F(\nabla^k u)\d x+\int_{\Omega'} F^\infty(\D_s^k u)
\end{align}
\end{proof}

We record the following Corollary
\begin{corollary}
Assume the conditions of Theorem \ref{thm:repr} hold. Then for any $t\in[0,1]$,
\begin{align}\label{eq:YoungConseq}
\begin{cases}
\langle \nu_x,F\rangle = F(\overline\nu_x) &\mathscr L^n-\text{a.e. } x\\
\langle\nu_x^\infty,F^\infty\rangle = F^\infty(\overline\nu_x^\infty) &\lambda-\text{a.e. }x \\
F\left(\overline\nu_x+t\overline\nu_x^\infty\frac{\d\lambda}{\d x}\right) = F(\overline\nu_x)+t F^\infty\left(\overline\nu_x^\infty,\frac{\d\lambda}{\d x}\right) \quad&\mathscr L^n-\text{a.e. } x
\end{cases}
\end{align}
Moreover, either $\lambda$ is purely singular or
\begin{align*}
F'(\nabla^k u)\cdot \overline\nu_x^\infty = F'(\overline\nu_x)\cdot\overline\nu_x^\infty
\end{align*}
\end{corollary}
\begin{proof}
Due to the result of Theorem \ref{thm:repr} we must have equality in all the calculations. In particular, we deduce \eqref{eq:YoungConseq} with $t=1$. The case $t\in[0,1]$ is now a direct consequence of the convexity of $F(\cdot)$.

For the moreover part, we differentate the third line of \eqref{eq:YoungConseq} at $t=1^-$ to find
\begin{align*}
F'(\nabla^k u)\cdot \overline\nu_x^\infty\frac{\d\lambda}{\d x} = F'(\overline\nu_x)\cdot\overline\nu_x^\infty\frac{\d\lambda}{\d x}.
\end{align*}
This implies the claim.
\end{proof}

\section{Constrained extended real-valued integrands}\label{sec:constr}
For the convenience of the reader we recall the relevant set-up.
We consider the following problem:
Let $\Omega\subset\R^n$ be an open and bounded subset of $\R^n$ and $g\in \WW^{k,1}(\Omega)$.
Let $K\subset \WW^{k,1}(\Omega)=\WW^{k,1}(\Omega,\R^N)$ be  a closed convex subset. 
We consider the functional
\begin{align}\label{eq:defconstr}
\F(v,\Omega)=\int_\Omega F(\nabla^k v)\d x
\end{align}
defined on $K\cap \WW^{k,1}_g(\Omega)$. We assume that $F(\cdot)$ is a convex, extended real-valued integrand satisfying moreover the explicit lower bound
\begin{align}\label{H1Constr}
F(\xi)\geq \theta(|\xi|)\tag{H1}
\end{align}
for all $\xi\in \M$ where $\theta\colon[0,\infty)\to [0,\infty)$ is an increasing convex function satisfying
\begin{align*}
\frac{\theta(t)}{t}\to \infty \text{ as } t\to\infty.
\end{align*}

The main theorem of this section is
\begin{theorem}\label{thm:mainConstr}
Suppose $F\colon \M\to \R\cup\{\infty\}$ is convex, lower semi-continuous, proper, essentially smooth and satisfies \eqref{H1Constr}. Let $g\in \WW^{k,1}(\Omega)\cap K$ with $F(s\nabla^k g)\in \LL^1(\Omega)$ for some $s>1$. Then minimisers $u\in \WW^{k,1}_g(\Omega)\cap K$ are characterised by the conditions
\begin{align*}
F^\ast(F'(\nabla^k u)) \in \LL^1(\Omega),\qquad \langle F'(\nabla^k u),\nabla u)\in \LL^1(\Omega)
\end{align*}
and
\begin{align*}
\int_\Omega F'(\nabla^k u)\nabla^k(u-v)\geq 0 \text{ for all } v\in K \text{ such that } u-v\in W^{k,\infty}_0(\Omega).
\end{align*}
\end{theorem}
We point out that, in the case where the constraint corresponds to an obstacle problem, we may infer that, under the assumptions of Theorem \ref{thm:mainConstr}, the Euler-Lagrange inequality for $\F(\cdot,\Omega)$ holds in the following strong sense:
\begin{corollary}\label{cor:mainConstr}
Suppose $\Omega$ is a $\WW^{1,1}$-extension domain and that $K$ is of the form $${K=\{u\in \LL^1(\Omega)\colon u\geq \psi \text{ a.e. in }\Omega\}}$$ where $\psi\in \WW^{k,1}(\Omega)$ is such that $F(\pm \nabla^k \psi)\in \LL^1(\Omega)$.
Under the assumptions of Theorem \ref{thm:mainConstr} we have for a minimiser $u\in \WW^{k,1}_g(\Omega)\cap K$ that $\langle F'(\nabla^k u),\nabla^k(v-u)\rangle\in \LL^1(\Omega)$ and
\begin{align*}
\int_\Omega \big\langle F'(\nabla^k u),\nabla^k( v-u)\big\rangle \d x\geq 0
\end{align*}
for all $v \in \WW^{k,1}_g(\Omega)\cap K$ satisfying the integrability conditions
\begin{align*}
F(-\nabla^k v)\in \LL^1(\Omega) \quad \text{ and } \quad F(\nabla^k v)\in \LL^1(\Omega).
\end{align*}
\end{corollary}

We remark that in the unconstrained case ($K=\WW^{k,1}(\Omega)$ by the same proof we have the following variant of Corollary \ref{cor:mainConstr}.
\begin{corollary}\label{cor:main}
Suppose $F\colon \M\to \R\cup\{\infty\}$ is convex, lower semi-continuous, proper, essentially smooth and satisfies \eqref{H1Constr}. Let $g\in \WW^{k,1}(\Omega)$ with $F(s\nabla^k g)\in \LL^1(\Omega)$ for some $s>1$. We have for a minimiser $u\in \WW^{1,1}_g(\Omega)$ that $\langle F'(\nabla^k u),\nabla\phi\rangle\in \LL^1(\Omega)$ and
\begin{align*}
\int_\Omega \big\langle F'(\nabla^k u),\nabla \phi\big\rangle \d x=0
\end{align*}
for all $\phi\in \WW^{1,1}(\Omega)$ with compact support contained in $\Omega$ satisfying the integrability conditions
\begin{align*}
F(-\nabla^k\phi)\in \LL^1(\Omega) \quad \text{ and } \quad F(\nabla^k\phi)\in \LL^1(\Omega).
\end{align*}
\end{corollary}

\begin{proof}[Proof of Theorem \ref{thm:mainConstr}]
Using Lemma \ref{lem:approximations} we obtain $F_s$ such that $F_s\nearrow F$ pointwise and locally uniformly on $\tp{dom}(F)$. Further $F'_s\to F'$ pointwise and locally uniformly on $\tp{dom}(F)$.

We begin by proving that
\begin{align*}
f_j\equiv\inf\left\{\int_\Omega F_j(\nabla^k v)\d x\colon v\in \WW^{k,1}_g(\Omega)\cap K\right\}\nearrow \int_\Omega F(\nabla^k u)\d x\equiv f.
\end{align*}
Clearly $f_j$ is an increasing sequence and we may take $v_j\in \WW^{k,1}_g(\Omega)\cap K$ such that
\begin{align*}
\int_\Omega F_j(\nabla^k v_k)\d x<\frac 1 j+f_j.
\end{align*}
Considering \eqref{H1Constr} and $f_j\leq f<\infty$, $(v_j)$ is bounded in $\WW^{k,1}(\Omega)$. In particular, we may extract a subsequence, not relabelled, such that
\begin{align*}
\nabla^i v_j\overset{\ast}{\rightharpoonup}\nabla^i v \text{ weakly}\ast \text{ in } BV(\Omega) \text{ and } \nabla^k v_j\overset{Y}{\rightharpoonup}\nu \text{ as } \LL^1-\text{Young-measures}.
\end{align*}
where $v\in BV(\Omega)$, $i\in\{0,\ldots,k-1\}$ and $\nu=((\nu_x)_{x\in\Omega},\lambda,(\nu_x^\infty)_{x\in\overline\Omega})$ is a generalised Young-measure.
 We remark that an implication is that, for any integer $s>1$,
\begin{align*}
\int_\Omega \big\langle F_s,\nu_x\big\rangle\d x+\int_{\overline\Omega} \big\langle F_s^\infty,\nu_x^\infty\big\rangle\d \lambda = \lim_{j\to\infty} \int_\Omega F_s(\nabla^k v_j)\d x\leq \lim_{j\to\infty} F_j(\nabla^k v_j)\d x.
\end{align*}
Recalling that $F_j^\infty\nearrow F^\infty$, by monotone convergence, we may deduce that
\begin{align*}
\lambda = 0 \text{ and } \int_\Omega \big\langle F,\nu_x\big\rangle\d x\leq \lim_{j\to\infty} f_j.
\end{align*}
Using standard results, see \cite{Alibert1997}, we deduce that $(\nabla^k v_j)$ is equi-integrable on $\Omega$, so that $v_j\rightharpoonup v$ weakly in $\WW^{k,1}(\Omega)$, where $v\in \WW^{k,1}_g(\Omega)\cap K$ by Mazur's Lemma and since $K$ is closed. By another standard result, the centre of mass of $\nu_x$ is $\nabla^k v(x)$ for almost all $x$. But now Jensen's inequality and the above allow us to conclude
\begin{align*}
\int_\Omega F(\nabla^k v)\d x\leq \lim_{j\to\infty} f_j
\end{align*}
and thus by minimality of $u$ we have proven our claim.

In particular, we may now write
\begin{align*}
\int_\Omega F_j(\nabla^k u)\d x\leq \int_\Omega F(\nabla^k u)=\e_k^2+\inf\left\{\int_\Omega F_k(\nabla^k v)\colon v\in \WW^{k,1}(\Omega)\cap K\right\},
\end{align*}
where $\e_k\searrow 0$. By Ekeland's variational principle we can find $u_j\in \WW^{k,1}_g(\Omega)\cap K$ such that
\begin{align*}
\int_\Omega |\nabla^k u-\nabla^k u_j|\d x\leq \e_j \quad \int_\Omega F_j(\nabla^k u_j)\d x\leq \int_\Omega F_j(\nabla^k u)\d x
\end{align*}
and
\begin{align*}
\int_\Omega F'_j(\nabla^k u_j)[\nabla^k(v-u_j)]\d x\geq -\e_j\int_\Omega |\nabla(v-u_j)|\d x.
\end{align*}
for all $v\in \WW^{k,1}_g(\Omega)\cap K$.

We put $\sigma_j = F'_j(\nabla^k u_j)$ and $\sigma = F'(\nabla^k u)$. Note that, as $F(\nabla^k u)\in \LL^1(\Omega)$, it must hold that ${\nabla^k u\in \tp{dom}(F)}$ almost everywhere in $\Omega$. In particular, the second definition makes sense, since $F$ is essentially smooth.

Note that $\sigma_j\in \LL^\infty(\Omega)$ (as $F_j$ is Lipschitz-continuous) and furthermore
\begin{align}\label{eq:dualityRelationConstr}
\langle \sigma_j,\nabla u_j\rangle = F^\ast_j(\sigma_j)+F_j(\nabla^k u_j)
\end{align}
holds almost everywhere in $\Omega$. In particular, this implies that $F^\ast_j(\sigma_j)\in \LL^1(\Omega)$. We further comment that $F^\ast_j$ is an extended real-valued, lower semi-continuous and convex integrand. Finally
\begin{align}\label{eq:convMeasureSigmaConstr}
\sigma_j\to \sigma \text{ in measure on } \Omega.
\end{align}
To reach this conclusion, we again use that $\nabla^k u\in \tp{dom}(F)$ almost everywhere in $\Omega$. It is also not difficult to check that
$F^\ast_j(\sigma_j)\searrow F^\ast(\xi)$ as $j\nearrow \infty$ pointwise in $\xi$ and consequently we deduce that
\begin{align}\label{eq:convMeasureFConstr}
F^\ast_j(\sigma_j)\to F^\ast(\sigma) \text{ in measure on } \Omega.
\end{align}

We note that
\begin{align*}
F^\ast_k(\xi)\geq F^\ast(\xi) \geq r|\xi|
\end{align*}
for all $\xi\in \M$ and $k>1$. Now as $g\in \WW^{k,1}(\Omega)\cap K$,
\begin{align*}
\int_\Omega \big\langle \sigma_j, \nabla u_j\big\rangle \d x\leq \int_\Omega \big\langle \sigma_j,\nabla g\big\rangle\d x+\e_j\int_\Omega |\nabla^k u_j-\nabla^k g|\d x,
\end{align*}
so that using \eqref{eq:dualityRelationConstr} we find, recalling that $F^\ast_j(\sigma_j)\in \LL^1(\Omega)$,
\begin{align*}
\int_\Omega F^\ast_j(\sigma_j)\d x\leq& \int_\Omega \big\langle \sigma_j,\nabla g\big\rangle+\e_j|\nabla^k u_j-\nabla^k g|-F_j(\nabla^k u_j)\d x\\
\leq& \int_\Omega \frac 1 s \langle \sigma_j,s\nabla g\rangle + 1\d x\\
\leq& \int_\Omega \frac 1 2 F^\ast_j(\sigma_j)+\frac 1 s F_j(s\nabla g)+1\d x.
\end{align*}
In particular,
\begin{align*}
r\int_\Omega |\sigma_j|\d x\leq \int_\Omega F^\ast_j(\sigma_j)\d x\leq \frac 1 {s-1}\int_\Omega F(s\nabla g)+s.
\end{align*}
Thus we obtain equi-integrability of $(\sigma_j)$ on $\Omega$ and hence \eqref{eq:convMeasureSigmaConstr} and Vitali's convergence theorem ensure that $\sigma\in \LL^1(\Omega)$ and
\begin{align*}
\sigma_j\to \sigma \text{ strongly in } \LL^1(\Omega).
\end{align*}

Using the duality relation, \eqref{eq:convMeasureFConstr} and Fatou's lemma, we deduce that ${\langle \sigma,\nabla u\rangle\in \LL^1(\Omega)}$. Consequently $\langle \sigma_j,\nabla u_j\rangle\to \langle \sigma,u\rangle$ in $\LL^1(\Omega)$. Further using Fatou's lemma, $F^\ast(\sigma)\in \LL^1(\Omega)$. 

As a consequence, we deduce that for any $v\in K$ such that $u-v\in\WW^{k,\infty}_0(\Omega)$,
\begin{align*}
\int_\Omega F'(\nabla u)[\nabla(v-u)]\geq 0.
\end{align*}
Thus the proof is complete.
\end{proof}

\begin{proof}[Proof of Corollary \ref{cor:mainConstr}]
Consider $u\in \WW^{k,1}_g(\Omega)\cap K$, a minimiser. By the main theorem, $u$ is an energy-extremal and $F^\ast(F'(\nabla^k u))\in \LL^1(\Omega)$. Since $u$, is energy-extremal we have that
\begin{align*}
\int_\Omega \big\langle \sigma,\nabla^k(v-u)\big\rangle\d x\geq 0
\end{align*}
for all $v\in K$ such that $u-v\in \WW^{k,\infty}_0(\Omega)$, where for brevity we write $\sigma = F'(\nabla^k u)$. Fix $v\in \WW^{k,1}_g(\Omega)\cap K$ such that $F(\pm \nabla^k v)\in \LL^1(\Omega)$. Extend $v$ and $\psi$ to $\WW^{k,1}(\R^n,\R^N)$ functions still denoted $v$ and $\psi$ respectively. Let $(\Psi_\e)_{\e>0}$ be a family of standard smooth mollifier and note that for sufficiently small $\e>0$, $\psi\star\Psi_\e,v\star \Psi_\e\in C^\infty_c(\Omega,\R^N)$. Further note by direct calculation that 
\begin{align*}
v\star\Psi_\e-\psi\star\Psi_\e+\psi\geq 0 \text{ a.e. in } \Omega.
\end{align*}
In particular for such $\e>0$,
\begin{align}\label{eq:ELConstr}
\int_\Omega \big\langle \sigma,\nabla(v\star\Psi_\e-\psi\star\Psi_\e+\psi-u)\big\rangle\d x \geq 0.
\end{align}
Now using Young's inequality,
\begin{align}\label{eq:absValueConstr}
\langle \sigma,\pm\nabla(v\star\Psi_\e)\rangle \leq F^\ast(\sigma)+F(\pm \nabla^k(v\star\Psi_\e))\nonumber\\
\langle \sigma,\pm\nabla(\psi\star\Psi_\e)\rangle \leq F^\ast(\sigma)+F(\pm \nabla^k(\psi\star\Psi_\e))
\end{align}
almost every in $\Omega$. In particular, combining these inequalities,
\begin{align*}
\left|\langle \sigma,\pm\nabla(v\star\Psi_\e)\rangle\right| \leq F^\ast(\sigma)+F(\nabla^k(v\star\Psi_\e))+F(-\nabla^k(v\star\Psi_\e))\nonumber\\
\left|\langle \sigma,\pm\nabla(\psi\star\Psi_\e)\rangle\right| \leq F^\ast(\sigma)+F(\nabla^k(\psi\star\Psi_\e))+F(-\nabla^k(v\star\Psi_\e))
\end{align*}
Using Jensen's inequality and Fatou's inequality, we deduce in a routine manner that $F(\pm \nabla^k(v\star\Psi_\e))\to F(\pm\nabla^k v)$ strongly in $\LL^1(\Omega)$ as $\e\to 0$ as well as that we have $F(\pm \nabla^k(\psi\star\Psi_\e)\to F(\pm \nabla^k \psi)$ strongly in $\LL^1(\Omega)$. In particular, using \eqref{eq:absValueConstr}, Fatou's lemma and Vitali convergence theorem,
\begin{align*}
\langle \sigma,\nabla(v\star\Psi_\e-\psi\star\Psi_\e+\psi-u)\rangle \to \langle \sigma,\nabla v-u\rangle
\end{align*}
as $\e\searrow 0$. Hence we may pass to the limit in \eqref{eq:ELConstr} to conclude the proof.
\end{proof}

\section{BV-minimisers}\label{sec:BV}
For the convenience of the reader we recall the set-up we work with in this section.
We consider the functional
\begin{align*}
\F(v)=\int_\Omega F(\nabla^k v)\d x
\end{align*}
defined on $\WW^{k,1}_g(\Omega)$. We assume that $F(\cdot)$ is a convex, extended real-valued integrand satisfying moreover the linear bound,
\begin{align}
F(\xi)\geq a(\xi)+c|\xi|\tag{H2}
\end{align}
for some linear $a(\cdot)$ and $c>0$.

We extend $\F(\cdot)$ to $\BV^k(\Omega)$ as follows:
Let $\Omega\Subset\Omega'$ and $g\in W^{k,1}(\Omega')$ such that $\int_{\Omega'\setminus\Omega} F(\nabla^k g)\d x<\infty$. We then define for $u\in \BV^k(\Omega)$,
\begin{align*}
\F(u)&=\inf\big\{\liminf \int_{\Omega'}F(\nabla^k u_j)\d x\colon (u_j)\in X\big\}\\
X&=\{(u_j)\colon u_j\in \WW^{k,1}(\Omega'),\nabla^i u_j\wstar \nabla^i u \text{ in } \BV(\Omega) \text{ for } i=0,\ldots,k-1,\\
&\qquad u_j\to g \text{ in } \WW^{k,1}(\Omega'\setminus\Omega)\}.
\end{align*}

We prove the following result:
\begin{theorem}
Let $\Omega$ be a Lipschitz domain. Suppose that $F(\cdot)$ is strictly convex and satisfies \eqref{eq:demiCoerc}. Let $u\in \BV^k(\Omega)$ be a minimiser of $\F(\cdot)$. Then $F'(\nabla^k u)$ is divergence-free in the sense of distributions. Moreover $\langle F'(\nabla^k u,\overline\nu_x\rangle\in L^1(\Omega')$
\end{theorem}
\begin{proof}
Let $\{F_j\}$ be the family of approximations to $F$ constructed in Lemma \ref{lem:approximations}. 
We first show that
$$
f_j\equiv \inf\{\F_j(v)\colon v\in \WW^{k,1}(\Omega)\}\nearrow \F(u)=f.
$$
Clearly $\{f_j\}$ is increasing in $j$ and $f_j\leq f$.
Pick $v_j\in \WW^{k,1}(\Omega')$ such that $v_j\to g$ in $\WW^{k,1}(\Omega'\setminus\Omega)$ and $\F_j(v_j)\leq f_j+\frac 1 j$. Due to \eqref{eq:demiCoerc}, $(v_j)$ is bounded in $\WW^{k,1}(\Omega')$ and we may extract a subsequence such that $\nabla^i v_j\wstar \nabla^i v$ in $\BV(\Omega')$ for $i=0,\ldots,k-1$ and some $v\in \BV^k(\Omega')$, and $\nabla^k v_j\overset{Y}\to\nu$ as $L^1$-Young measures where $\nu=((\nu_x)_{x\in\Omega'},\lambda,(\nu_x^\infty)_{x\in\overline{\Omega'}})$. We further have for any integer $s>0$, using Theorem \ref{thm:repr},
\begin{align*}
\lim_{j\to\infty} f_j \geq \lim_{j\to\infty} \int_{\Omega'} F_s(\nabla v_j)\d x=\int_{\Omega'} \langle F_j,\nu_x\rangle\d x+\int_{\overline\Omega'} \langle F_j^\infty,\nu_x^\infty\rangle\d\lambda=\F_k(v).
\end{align*}
Taking $k\to\infty$, we deduce that $\F(v)\leq \F(u)$ and hence by minimality of $u$ we have the required claim.

Take $u_l\in \WW^{k,1}(\Omega')$ such that $\nabla^i u_l\wstar \nabla^i u$ in $\BV(\Omega)$ for $i=0,\ldots,k-1$, $u_l\to g$ in $\WW^{k,1}(\Omega'\setminus\Omega)$ and $\F(u_l)\leq \F(u)+\frac 1 l$. Let $\{\Psi_s\}$ be the family of diffeomorphisms defined in Lemma \ref{lem:diffeos}. Set $u_l^s(x)=u_l(\Psi_s(x)$ and $g^s(x)=g(\Psi_s(x))$. Note that then for any fixed $i$,
\begin{align*}
\int_{\Omega'} F_i(\nabla^k u_l^s)\d x\to \int_{\Omega'} F_i(\nabla^k u_l)\d x
\end{align*}
as $s\nearrow 1$. Further note that $u_l^s\to g^s$ in $W^{k,1}(\Omega'\setminus \Psi_s(\Omega))$. Let $\eta_s$ be a smooth cut-off function supported on $\Omega'\setminus \Psi_s(\Omega)$, with $\eta_s=1$ on $\Omega'\setminus\Omega$, and $|\nabla^j\eta_s|\leq c(n)d(\p\Omega,\p \Psi_s(\Omega))^{-j}$ for $j=1,\ldots,k$. Introduce 
$$
w_l^s = u_l^s+\eta_s(g^s-u_l^s).
$$
Note that $w_l^s = g^s$ near $\p\Omega'$.

We estimate
\begin{align*}
\left|\int_{\Omega'} F_i(\nabla^k w_l^s)-F_i(\nabla^k u_l^s)\d x\right|\leq& i \sum_{j=1}^{k}\int_{\Omega\setminus \Psi_s(\Omega))} \binom{k}{j} |\nabla^j \eta| |\nabla^{k-j}(g^s-u_l^s)|\d x\\
&\quad+i\int_{\Omega'\setminus\Psi_s(\Omega)} |\nabla^k(g^s-u_l^s)|\d x \coloneqq A.
\end{align*}
Noting that $\|\nabla^{k-j}g^s-u_l^s\|_{L^1(\Omega\setminus\Psi_s(\Omega)}\leq c |\Omega\setminus\Psi_s(\Omega)|^j\|\nabla^k(g^s-u_l^s)\|_{L^1(\Omega\setminus\Psi_s(\Omega)}$ and that $|\Omega\setminus\Psi_s(\Omega)|\sim d(\p\Omega,\p\Psi_s(\Omega))$, we deduce that $A\to 0$ as $l\to \infty$. Hence
\begin{align*}
\F_i(w_l^s)\to \F_i(u_l^s)
\end{align*}
as $l\to\infty$. Extracting a diagonal subsequence we have found a sequence $u_l\in \WW^{k,1}(\Omega')$ such that $\nabla^i u_l\wstar \nabla^i u$ in $BV(\Omega)$ for $i=0,\ldots,k-1$, $u_l\to g$ in $\WW^{k,1}(\Omega)$, $u_l = g^{s_l}$ on $\p\Omega'$ in the sense of traces where $s_l\nearrow 1$ as $l\to\infty$ and
\begin{align*}
\F(u_l)\leq \F(u)+\frac 2 l.
\end{align*}

We may thus write
\begin{align*}
\int_{\Omega'} F_i(\nabla^k u_l)\d x\leq \int_{\Omega'} F(\nabla^k u_l)\d x\leq \F(u)+\frac 2 l=f_i+\frac 2 l+\e_i^2
\end{align*}
where $\e_\searrow 0$. By Ekeland's variational principle, we obtain $w_l^i\in W^{k,1}(\Omega)$ such that with $\e_l^i=\sqrt{\frac 2 l+\e_i^2}$,
\begin{align}\label{eq:Ekeland}
\int_{\Omega'}|\nabla^k w_l^i-\nabla^k u_l|\d x\leq \e_l^i, \quad \int_{\Omega'} F_i(\nabla^k w_l^i)\d x\leq \int_{\Omega'} F_i(\nabla^k u_l)\d x
\end{align}
and
\begin{align*}
\int_{\Omega'} F'_i(\nabla^k w_l^i)[\nabla^k\phi]\d x\leq \e_l^i\int_\Omega |\nabla^k\phi|\d x
\end{align*}
for all $\phi\in W^{k,1}_0(\Omega')$.

Note that
\begin{align*}
\F_i(v_i)\leq \F(w_l^i)\leq \F(u)+(\tilde\e_l^i)^2.
\end{align*}
Letting $i,l\to\infty$, we see that $\F(w_l^i)\to \F(u)$. In particular, we deduce, using also \eqref{eq:Ekeland}, that there is a subsequence, denoted $w_j$, such that
\begin{align*}
\nabla^i w_j\wstar \nabla^i u \text{ in } \BV(\Omega') \text{ for } i=0,\ldots,k-1, \, \nabla^k w_j\overset{Y}\rightarrow \nu \text{ as } L^1-\text{Young measures}.
\end{align*}
Moreover, we have that $\nabla w_j^k\to\overline\nu_x$ almost everywhere in $\Omega'$, see Lemma 2.10 in \cite{Kristensen2020}. In particular, $\sigma_j=F'(\nabla^k w_j)\to \sigma = F'(\overline\nu_x)$ almost everywhere in $\Omega'$.

By Lipschitz continuity of $F_j$, $\sigma_j\in L^\infty(\Omega')$ and we have the extremality relation
\begin{align*}
\langle \sigma_j,\nabla^k w_j\rangle = F^\ast_j(w_j)+F_j(\nabla^k w_j),
\end{align*}
valid almost everywhere in $\Omega'$. It is not hard to check that $F^\ast_j(\xi)\searrow F^\ast(\xi)$ as $j\to\infty$ pointwise in $\xi$, and thus, in particular, after adding an affine function to $F$ if necessary,
\begin{align}\label{eq:lowerBound}
F_j^\ast(\xi)\geq r |\xi|
\end{align}
for some $r>0$. We further note that
\begin{align*}
\int_{\Omega'} \big\langle \sigma_j, \nabla^k w_j\big\rangle\d x\leq \int_{\Omega'} \big\langle \sigma_j,\nabla^k g^{s_j}\big\rangle\d x+\tilde \e_j\int_{\Omega'}|\nabla^k w_j-\nabla^k g^{s_j}|\d x,
\end{align*}
where $\tilde \e_j=\e_{l_j}^{i_j}$. In particular, we may estimate
\begin{align*}
\int_{\Omega'} F_j^\ast(\sigma_j)\d x\leq& \int_{\Omega'} \big\langle \sigma_j,\nabla^k g^{s_j}\big\rangle\tilde+ \e_j|\nabla^k w_j-\nabla^k g^{s_j}|\d x-F_j(\nabla w_j)\d x\\
\leq& \int_{\Omega'} \frac 1 s\big\langle \sigma_j,s\nabla^k g^{s_j}\rangle+1\d x\\
\leq& \int_{\Omega'} \frac 1 s F_j^\ast(\sigma_j)+\frac 1 s F_j(s\nabla^k g^{s_j})+1.
\end{align*}
We note that
\begin{align*}
\int_{\Omega'} F_j(s\nabla^k g^{s_j})\d x = \int_{\Omega'} |\Psi_{s_j}|F_j(s\nabla^k g)\d x\to \int_{\Omega'} F_j(s \nabla^k g)\d x
\end{align*}
as $j\to\infty$. In particular, we deduce
\begin{align*}
\int_{\Omega'} F^\ast_j(\sigma_j)\d x\leq \frac c {s-1}\int_{\Omega'} F(s\nabla^k g)+1\d x.
\end{align*}
Hence $(\sigma_j)$ is equi-integrable on $\Omega'$ and thus by Vitali's convergence theorem, $\sigma_j\to \sigma$ in $L^1(\Omega')$. We deduce that
$\Div \sigma=0$ in the sense of distributions. 

By Fatou's lemma, $F^\ast(\sigma)\in L^1(\Omega')$ and using the duality relation
\begin{align*}
\langle \sigma_s,\nabla^k u_s\rangle = F^\ast_s(\sigma_s)+F_s(\nabla^k u_s)
\end{align*}
as well as again Fatou's lemma, we conclude that $\langle\sigma,\overline\nu_x\rangle\in L^1(\Omega')$.

Note that since $F$ is strictly convex, $\lambda$ is purely singular and consequently $\overline\nu_x = \nabla^k u$. This concludes the proof.
\end{proof}

\bibliographystyle{plain}
\bibliography{../bibtex/Euler}
\end{document}